\documentclass[10pt]{amsart}
\usepackage{amsmath}
\usepackage{amsfonts}
\usepackage{amssymb}
\usepackage{amsthm}
\usepackage{mathrsfs}       
\usepackage{graphicx}
\usepackage[utf8]{inputenc}
\allowdisplaybreaks

\usepackage{dsfont}

\usepackage{multicol}

\usepackage{xspace}
\usepackage[colorlinks=true]{hyperref}

\newcommand{\R}{\mathbb{R}}
\newcommand{\N}{\mathbb{N}}

\newcommand{\comp}{\mathbb{C}}

\newcommand{\C}{\mathscr{C}}

\newcommand{\ud}{\,\mathrm{d}}

\let \Re \relax
\DeclareMathOperator{\Re}{Re}
\let \Im \relax
\DeclareMathOperator{\Im}{Im}

\let \div \relax
\DeclareMathOperator{\div}{div}

\DeclareMathOperator{\supp}{supp}

\newcommand{\ovl}[1]{\overline{#1}}

\newtheorem{thm}{THEOREM}[section]
\newtheorem{remark}[thm]{REMARK}
\newtheorem{lemma}[thm]{LEMMA}

\newtheorem{proposition}[thm]{PROPOSITION}

\newcounter{thmbiss}

\author[G.~Lebeau]{Gilles Lebeau}
\begin{address}{Gilles Lebeau,
Laboratoire J.A. Dieudonné,
UMR CNRS 7351,
Université de Nice Sophia-Antipolis,
Parc Valrose,
06108 Nice Cedex 02,
France}
  \email{\href{mailto:gilles.lebeau@unice.fr}{gilles.lebeau@unice.fr}}
 \end{address}

\author[I.~Moyano]{Iv\'an Moyano}
\begin{address}{Iv\'an Moyano,
DPMMS, Centre for Mathematical Sciences, University of Cambridge, Wilberforce Road, CB3 0WA, United Kingdom}
  \email{\href{mailto:im449@dpmms.cam.ac.uk}{im449@dpmms.cam.ac.uk}}
 \end{address}

\title[Spectral inequalities for the Schrödinger operator]{Spectral inequalities for the Schrödinger operator}

\begin{document}

\maketitle

\begin{abstract}

In this paper we deal with the so-called ``spectral inequalities'’, which yield a sharp quantification of the unique continuation for the spectral family associated with the 
Schrödinger operator in $ \mathbb{R}^d$
\begin{equation*}
H_{g,V} = \Delta_g + V(x), 
\end{equation*}
where $\Delta_g$ is the 
Laplace-Beltrami operator   with respect to an analytic metric $g$,  which is a perturbation of the Euclidean metric,  
and $V(x)$ a real valued analytic potential  vanishing at infinity.
\end{abstract}

\tableofcontents

\section{Introduction and main result}
\label{sec: Introduction}

Let $g$ be 
a Riemannian metric 
on $\R^d$, $\Delta_g$  the usual Laplacian in the metric $g$, and 
$V= V(x)$ a real potential function, not necessarily short range, such as the typical examples of long-range interactions in scattering theory (cf. \cite[vol.IV Ch.XXX]{HormanderVol4}).
In this paper we prove spectral inequalities for the Schrödinger operator 
\begin{equation*}
H_{g,V} := -\Delta_g + V(x), \qquad \textrm{ in } \R^d,  \quad d \geq 1.
\end{equation*}   Our approach relies on interpolation inequalities, in the spirit of the works \cite{LeRo,JeLe,LeZu}, but adapted to the unbounded case. We will  use  spectral projectors, holomorphic extension arguments,  and suitable interpolation estimates for holomorphic functions.\par 

In the case $V=0$, we recover some classical quantifications of the uncertainty principle dues to Zigmund \cite[pp.202-208]{Zigmund}, Logvenenko and Svereda \cite{LogSereda} and Kovrojkine \cite{KovrojkinePaper,KovrojkineSurvey}, among other references (see Section \ref{sec: uncertainty literature} for further details).

\subsection{Geometric conditions for the observability sets}
Given $R>0$ and $x\in \R^d$, we denote by $B^g_R(x)$ the ball of radius $R$, with respect to the metric $g$ centered in $x$. When $x=0$, we simply write $B^g_R$ and if moreover $g=Id$, one can simply write $B_R$ as usual. We shall work with Lebesgue measurable sets $\omega \subset \R^d$  satisfying the condition
\begin{equation}
\exists R,\delta >0 \quad \textrm{ such that } \quad 
\frac{mes(\omega \cap B^g_R(x))}{mes(B^g_R(x))} \geq \delta, \quad \forall x \in \R^d.
\label{eq:geometriccondition}
\end{equation} Since in the present work the metric $g$ will be asymptotically the flat metric,  (\ref{eq:geometriccondition}) can be replaced by the same condition with the Euclidean metric: 
\begin{equation}
\exists R,\delta > 0, \quad \textrm{ such that } \quad \inf_{x\in \R^d} mes\left\{ t\in \omega, \, |x - t| < R  \right\} \geq \delta.
\label{eq:geometricconditiontechnical}
\end{equation} 

\subsection{Main result: spectral inequality for the Schrödinger operator on $\R^d$ }

Let $g = g_{ij}(x)$ be a Riemannian metric in $\R^d$ and consider 
 the Laplace-Beltrami operator associated to $g$, i.e., 
\begin{equation}
\Delta_g u = \frac{1}{\sqrt{ \det g }} \div \left( \sqrt{\det g} g^{-1} \nabla u  \right), 
\label{eq:LaplaceBeltrami}
\end{equation} where $g^{-1}(x) = (g^{ij})(x)$ denotes as usual the inverse metric of $g$. Given $V=V(x)$ a real-valued potential function, one defines the associated Schrödinger operator 
\begin{equation}
H_{g,V}:= -\Delta_g + V(x), \qquad  \textrm{ on } \ \  D(H_{g,V}),
\label{eq:Schrodinger operator}
\end{equation} where 
\begin{equation*}
D(H_{g,V})= \left\{ u \in L^2(\R^d; \sqrt{\det g} \ud x ); \, H_{g,V}(u) \in L^2(\R^d;\sqrt{\det g} \ud x )  \right\}\ .
\end{equation*} 

We will assume that the metric $g$ and the real-valued potential $V$  satisfy
 the following hypothesis: 
\begin{align}
& \textrm{the metric } g \textrm{ is analytic},
 \textrm{of the form } g = Id + \tilde{g} \    \textrm{with} \lim_{\vert x\vert \rightarrow \infty} \tilde{g}(x)=0,
  \label{eq: metric analytic}\\
& \textrm{the potential } V \textrm{ is analytic, real valued},  \textrm{with} \lim_{\vert x\vert \rightarrow \infty} V(x)=0,
\label{eq: potential analytic}\\
& \exists a >0 \textrm{ such that } g \textrm{ and }V \textrm{ extend holomorphically in the} 
\label{eq: metric and potential holomorphically extended}\\
& \textrm{complex open set defined by } U_a := \left\{ z \in \comp^d; \, |\Im(z)| < a   \right\}, \nonumber \\
& \exists \epsilon \in (0,1) \  \textrm{such that} \  \tilde{g} \ \textrm{ and} \ V  \ \textrm{ satisfy}  
\   \textrm{for}\  \vert \alpha\vert\leq 2  \label{eq: metric perturbation by a symbol}\\
 & | \partial^{\alpha} \tilde{g}(z) |+ | \partial^{\alpha} V(z) |
 \leq C_{\alpha} (1 + |z|)^{ - \epsilon - |\alpha| }, \quad \forall z \in U_a  .  \nonumber 
\end{align}

Under these hypothesis
 one can check (see Proposition \ref{prop:description of the spectrum}) that the Schrödinger operator $H_{g,V}$  is an unbounded self-adjoint operator.  In section \ref{sec: spectral analysis} we will recall some basic facts 
 on the functional calculus of self-adjoint operators. In particular, the spectral projectors $ \Pi_{\mu}(g,V)$ are
 defined in \eqref{eq:spectral projectors}. 
 In this article we will prove that the family of specral projectors $ \Pi_{\mu}(g,V)$
 enjoy a spectral inequality, i.e., an observability inequality on a set $\omega \subset \R^d$ for low frequencies as long as the observability set satisfies (\ref{eq:geometriccondition}). \par

\bigskip
Given $\mu\in \R$, let us introduce
\begin{equation}
\mu^{\frac{1}{2}}_{\pm}:= \sqrt{\mu}, \quad \mu\geq 0, \qquad \mu^{\frac{1}{2}}_{\pm}:= \pm i\sqrt{|\mu|}, \quad \mu <0.
\label{eq: squqre root of E}
\end{equation} 
Our main result is the following:

\begin{thm}
 Let $\omega \subset \R$ be a measurable set satisfying the geometric condition 
 (\ref{eq:geometricconditiontechnical}) and let $(g,V)$ satisfy hypothesis 
 (\ref{eq: metric analytic})--(\ref{eq: metric perturbation by a symbol}). Then, there exist constants $A=A(\omega,g,V)>0$ and $C=C(\omega,g,V)>0$ such that for all $ \mu \in \R $ and all $ f \in L^2(\R^d)$, one has 
\begin{equation}
\left\| \Pi_{\mu}(g,V) f \right\|_{L^2(\R^d;\sqrt{\det g} \ud x)} \leq   A \ 
| e^{C \mu^{\frac{1}{2}}_{\pm} } | \  \| \Pi_{\mu}(g,V) f \|_{L^2(\omega;\sqrt{\det g} \ud x )}. 
\label{eq:spectral inequality potential}
\end{equation} 
\label{thm: Spectral inequality Potential}
\end{thm}

\subsection{A special case: spectral inequality for the Laplacian operator in $\R^d$}
When $g= Id$ and $V=0$ as $H_0 = -\Delta_x$ is  the usual flat Laplacian, it is well-known that  that $\sigma(Id, 0) = [0, \infty)$
is  a purely absolutely continuous spectrum. Furthermore, the spectral projectors are explicitly determined through the Fourier transform, i.e., 

\begin{equation*}
\Pi(Id,0)_{\mu}f = \frac{1}{ (2\pi)^d }\int_{|\xi| < \mu} \hat{f}(\xi) e^{i x \cdot \xi} \ud x, 
\end{equation*} recalling that the classical Fourier transform is defined by
\begin{equation*}
\hat{g}(\xi) := \int_{\R^d} g(x) e^{-i x \cdot \xi} \ud x, \quad \forall \xi \in \R^d, \quad \forall g\in L^2(\R^d). 
\end{equation*} As a result, one can recast (\ref{eq:spectral inequality potential}) as the following familiar spectral inequality.

\begin{thm}
Let $\omega \subset \R$ be a Lebesgue measurable set satisfying the geometric condition (\ref{eq:geometricconditiontechnical}). Then, there exist constants
 $A= A(\omega)>0, C=C(\omega)>0$ such that  for all $ \mu \in \R_+$ and all $ f \in L^2(\R^d)$, one has 
\begin{equation}
\| f \|_{L^2(\R)} \leq A(\omega) e^{C(\omega)\mu} \| f \|_{L^2(\omega)}, \quad \textrm{whenever} \quad \supp \hat{f} \subset \ovl{B_\mu}.
\label{eq:spectarlinequality}
\end{equation} 
\label{thm: Spectral inequality}
\end{thm}

\section{Context of our results and previous works}

\subsection{Spectral inequalities, Logvenenko-Sereda inequalities and the uncertainty principle}

\subsubsection{Spectral Inequalities}

Given a compact Riemannian manifold $M$ equipped with a metric $g$ the spectral inequalities for $V=0$ have been introduced by the first author and D. Jerison in \cite{JeLe} and also in \cite{LeZu}, among other works (see \cite{LeRLe} and the references therein). \par 

In the non-compact case when $M = \R^d$ with the usual Euclidean metric, the methods of \cite{LeRo} have been extended in \cite{LeRMoyano} in such a way that inequality (\ref{eq:spectarlinequality}) holds whenever $\omega \subset \R^d$ is an open set satisfying the geometric condition
\begin{equation}
\exists R>0, \delta >0 \textrm{ s.t. } \forall y \in \R^d, \, \exists y'\in \R^d \textrm{ with } B(y',R) \subset \omega \textrm{ and } d(y,y') < \delta.  
\label{eq: condition Kolmo}
\end{equation} This condition is thus proven to be sufficient for the null-controllability of the heat equation to hold, as well as for some hypoellipitic equations, like the Kolmogorov equation, arising in kinetic theory \cite{MoyanoSchwartz}. This condition is however far from being necessary (see Section \ref{sec: controllability heat whole space} for details).

\subsubsection{Quantification of the uncertainity principle and the Logvenenko-Sereda inequality}
\label{sec: uncertainty literature}

The classical uncertainty principle in harmonic analysis accounts for the fact that a function cannot be localised both in space and in the frequency variable (cf. \cite[Prop. 10.2, p. 270]{MuscaluSchlag}) : 
\begin{equation*}
\exists C>0, \quad  \left( \int_{\mathbb R} x^2 |f(x)|^2 dx  \right)^{\frac{1}{2}}  \left( \int_{\mathbb R} \xi^2 |\hat{f}(\xi)|^2 d \xi  \right)^{\frac{1}{2}} \geq C\| f \|^2_{L^2(\mathbb  R)}, \quad \forall f \in L^2(\R).
\end{equation*} A version of this inequality due to Amrein and Berthier (cf. \cite{AmreinBerthier}) guarantees that
\begin{equation*}
 \int_{\mathbb R \setminus \omega}  |f(x)|^2 dx  +  \int_{\mathbb R \setminus \Sigma}  |\hat{f}(\xi)|^2 d \xi 
 \geq C \| f \|^2_{L^2(\mathbb  R)}, \quad mes(\omega) + mes(\Sigma) < \infty.
\end{equation*} In the aim of giving more quantitative versions of the uncertainty principle in $\R^d$, one can find very significant literature (see \cite{KovrojkineSurvey} and the references therein for further details).  In particular, Logvinenko and Sereda proved in \cite{LogSereda} that the condition 
\begin{equation}
E \subset \R \textrm{ measurable s.t.} \, \exists \gamma >0,\,  a >0 \, \textrm{ s.t. } \, \frac{mes(E\cap I)}{mes(I)} \geq \gamma
\label{eq:LogSereda condition}
\end{equation} whenever $I$ is an interval of length $a$ is sufficient to ensure that 
\begin{equation}
\forall b\geq 0, \, \exists C = C(\gamma, a, b)>0 \textrm{ s.t.} \int_E |f(x)|^2 \ud x \geq C \|f\|^2_{L^2(\R)}, \textrm{ if } \supp(\hat{f}) \subset (-b, b) . 
\label{eq:LogSereda inequality}
\end{equation} On the other hand, the authors do not  get a sharp estimate  of $C$ with respect to the parameters $a,b,\gamma$. This was achieved  by Kovrijkine \cite{KovrojkinePaper}, where the author proves that 
\begin{equation*}
\exists K>0 \textrm{ such that } C(\gamma, a, b) = \left( \frac{\gamma}{K}  \right)^{K(ab + 1)},
\end{equation*} with $\gamma < K$. Moreover, the Logvenenko-Sereda inequality (\ref{eq:LogSereda inequality}) also holds in any $L^p(\R)$ space with $p\in[1,\infty]$. This is possible by combining the Bernstein's inequality, a suitable Remez-type inequality and some previous results by A. Zigmund in lacunary series \cite[pp. 202-208]{Zigmund}. The same results were obtained by Nazarov in \cite{Nazarov}. We refer to \cite{EgidiSurvey} and the references therein for more details on this subject. \par 

Thus our Theorem \ref{thm: Spectral inequality} (when $V=0$ and $g=Id$) recovers the Logvinenko-Sereda inequality in $\R^d$.

\subsection{Controllability of the parabolic equations in the whole space}
\label{sec: controllability heat whole space}
Some recent work has been concerned with the problem of characterising the sets having \textit{``good observability properties"} in the whole space in the context of the controllability of the heat equation and some other parabolic problems. More precisely, given an open set $\omega \subset \R^d$, for $d\geq 1$, let us consider the heat equation 
\begin{equation}
\left( \partial_t - \Delta_x \right) u = \chi_{\omega} g, \quad (t,x) \in \R_+ \times \R^d,
\label{eq:heat}
\end{equation} where $\chi_{\omega}$ is the characteristic function of $\omega \subset \R^d$ and $g$ is a forcing term, that we call a control, supported in $(0,T) \times \omega$. The small-time null-controllability of (\ref{eq:heat}) in an $L^2$ setting is equivalent to the following property
\begin{equation}
\forall T>0, \, \forall u_0 \in L^2(\R^d), \quad \exists g\in L^2((0,T)\times \omega) \quad \textrm{s.t} \quad  u|_{t=T} = 0,
\label{eq: controllability heat}
\end{equation} where $u$ is the solution of (\ref{eq:heat}) with $u|_{t=0} = u_0$. According to the classical HUM method, the null controllability of (\ref{eq:heat}) is equivalent to the observability for the adjoint system 
\begin{equation}
 \left( \partial_t - \Delta_x \right) \psi = 0, \quad (t,x) \in \R_+ \times \R^d,
\label{eq:heatAdjoint} 
\end{equation} which is equivalent to the following observability inequality:
\begin{equation}
\exists C_{obs}>0 \quad \textrm{s.t. } \forall \psi \in L^2(\R^d), \quad \int_{\R^d} |\psi|_{t=T}|^2 \ud x \leq \int_0^T \int_{\omega} |\psi(t,x)|^2 \ud t \ud x,
\label{eq:ObservabilityIntro}
\end{equation} where $\psi$ solves (\ref{eq:heatAdjoint}) with $\psi|_{t=0} = \psi_0$. \par 

In the recent works \cite{CanZhang, EgidiVeselic, Nakic1}, the authors use the Logvinenko-Sereda inequality to show that \eqref{eq:ObservabilityIntro} holds if and only if $\omega \subset \R^d$ is a measurable set satisfying (\ref{eq:geometriccondition}) for some $R, \delta > 0$.

\subsection{Outline of the paper} 
\label{sec: Outline}

In Section \ref{sec: carleman-interpolation} we give a proof of 
Proposition 
\ref{proposition:interpolation inequality multiD} which is 
a basic interpolation estimate for holomorphic functions defined in a tubular neighborhood of $\mathbb R^d$.
Although one can obtain this type of result like in \cite{KovrojkinePaper} using the Remez inequality
for polynomials as a starting point, we choose instead to present a proof which uses a Carleman estimate for functions
of one complex variable (see Lemma \ref{lem1}).\\
In section \ref{sec: spectral analysis} we recall some facts on the spectral theory of the Schr\" odinger operator and we
introduce the Poisson kernel.\\
In section \ref{sec: holomorphic extension multiD}, we prove estimates on the holomorphic
extension of solutions to the Poisson equation.\\
Finally, in section \ref{sec: End of the proof multiD}, we show that Theorem \ref{thm: Spectral inequality Potential}
is an easy consequence of the previous holomorphic extension estimates.


\section{Carleman estimates and interpolation inequalities}
\label{sec: carleman-interpolation}

Recall that for $a>0$ we denote by $U_a$ the tubular neighborhood of $\mathbb R^d$ in 
$\mathbb C^d$, 
$$U_a=\{z\in \mathbb C^d, \vert Im(z)\vert <a\}.$$
Let $\mathcal H_a$ be the Banach space of
holomorphic functions  $f(x+iy)$ in  $U_a$ such that

$$ \Vert f\Vert_{\mathcal H_a}:=\sup_{\vert y\vert <a} \Vert f(.+iy)\Vert_{L^2(\mathbb R^d)}<\infty.$$
By the Paley-Wiener Theorem, $\mathcal H_a$ is the space of Fourier Transforms of functions
$g\in L^2(\mathbb R^d)$ such that $\sup_{\vert\eta\vert <a} \Vert e^{\eta.y}g(y)\Vert_{L^2(\mathbb R^d)}<\infty$.\\

 The goal of this section is to prove the following result
 
 \begin{proposition}
Let $\omega \subset \R^d$ satisfying the density condition (\ref{eq:geometricconditiontechnical}). Then, there exist constants $C=C(a,\omega)>0$ and $\nu=\nu(a,\omega) \in ]0,1[$ such that 
\begin{equation}
\int_{\R^d} |f|^2 \ud x \leq C \left(  \int_{\omega} |f|^2 \ud x \right)^{\nu} \left(  \int_{U_a} |f|^2 |\ud z| \right)^{1 - \nu},
\label{eq:interpolation inequality multiD}
\end{equation} for any $f \in \mathcal H_a$.
\label{proposition:interpolation inequality multiD}
\end{proposition}

 Observe that Theorem \ref{thm: Spectral inequality} is an easy consequence of 
Proposition \ref{proposition:interpolation inequality multiD}. In fact, if $f\in L^2(\mathbb R^d)$ is such that its Fourier transform
 $\hat f(\xi)$ is supported in the ball  $\vert \xi\vert \leq \mu$, by the Fourier inversion formula one has
 
 $$ f(x)=(2\pi)^{-d}\int_{\vert \xi\vert \leq \mu}  e^{ix\xi} \hat f(\xi) d\xi, \qquad \forall x \in \R^d .$$
 
 Therefore $f$ is the restriction to $\mathbb R^d$ of the holomorphic function $f(z)$ defined on $\mathbb C^d$
 by  $ f(z)=(2\pi)^{-d}\int_{\vert \xi\vert \leq \mu} e^{iz\xi} \hat f(\xi) d\xi $, and one has by Plancherel theorem
 $$ \Vert f(.+iy)\Vert^2_{L^2(\mathbb R^d)}=(2\pi)^{-d}\int_{\vert \xi\vert \leq \mu} \vert e^{-y\xi}\hat f(\xi)\vert^2 d\xi
 \leq e^{2\mu \vert y\vert}\Vert f\Vert^2_{L^2(\mathbb R^d)}.$$
 Therefore, from \eqref{eq:interpolation inequality multiD} we get
 
 $$ \Vert f\Vert^2_{L^2(\mathbb R^d)} \leq C \left(  \int_{\omega} |f|^2 \ud x \right)^{\nu}
  \left( c_da^d e^{2\mu a}  \Vert f\Vert^2_{L^2(\mathbb R^d)} \right)^{1 - \nu},$$
where $c_d$ is the volume of the unit sphere in $\mathbb R^d$,  and this implies \eqref{eq:spectarlinequality}.\\
 
 We will prove Proposition \ref{proposition:interpolation inequality multiD} in several steps.
First  we prove suitable Carleman estimates for holomorphic functions of one complex variable. In particular, we prove the interpolation estimate given in Lemma \ref{lem2} below. We then deduce the 
multidimensional interpolation inequality given in Proposition \ref{prop2}. Finally, we get the
proof of Proposition \ref{proposition:interpolation inequality multiD} by a simple covering argument.

\subsection{Carleman estimates for $\ovl{\partial}$ in $\mathbb C$}
\label{sec: Carleman estimates}

In this section, we prove basic estimates for holomorphic functions of one complex variable.
We denote by $d\lambda$ the Lebesgue measure on $\mathbb C\simeq \mathbb R^2$. \\

Let $X\subset \mathbb C$ be an open bounded connected domain with regular boundary. 
We start with the following classical Carleman inequality.

\begin{lemma}\label{lem1}
Let $\varphi(x)$ be a continuous function on $X$ such that $\triangle \varphi=\nu$
is a Borel measure on $X$. For all $f\in C_0^\infty(X)$ and all $h>0$, the following inequality holds true

\begin{equation}\label{gl1} 
4h^2\int_{X}  e^{ 2\frac{\varphi(x)}{h} }\vert \ovl \partial f (x)\vert ^2 d\lambda
 \geq 
h  \int_X e^{ 2\frac{\varphi (x)}{h} }\vert  f(x)\vert^2 d\nu.
\end{equation} 
\end{lemma}
\begin{proof}
Let $f\in C_0^\infty (X)$ be given and let  $Q\subset X$ be a compact set such that the support of $f$ is contained in $Q$.
We first assume that $\varphi$ is smooth in a neighborhood of $Q$.
Let $P : = \frac{2h}{i} \ovl{\partial}$. We define the conjugate operator 
\begin{equation*}
P_{\varphi} := e^{ \frac{\varphi}{h} } P e^{ -\frac{\varphi}{h} }= A+iB, \quad A= \frac{h}{i}\partial_x-\varphi'_y=A^*, \
B= \frac{h}{i}\partial_y+\varphi'_x=B^*.
\end{equation*} 
Set $g=e^{\frac{\varphi}{h}}f\in C_0^\infty(Q)$. One has $P_\varphi g=e^{\frac{\varphi}{h}}Pf$ and  
by integration by part we get

\begin{equation}\label{gl2}
\begin{aligned} 
\Vert P_\varphi g\Vert^2_{L^2(X)} &=  \Vert Ag\Vert^2_{L^2(X)}+ \Vert B g\Vert^2_{L^2(X)}
+ i(Bg,Ag)_{L^2(X)} -i (Ag,Bg)_{L^2(X)}\\ 
&=\Vert Ag\Vert^2_{L^2(X)}+ \Vert B g\Vert^2_{L^2(X)}+i([A,B]g,g)_{L^2(X)}
\end{aligned}
\end{equation}
Since $[A,B]=AB-BA=-ih\triangle\varphi$, we get
\begin{equation}\label{gl3}
4h^2\int_{X}  e^{ 2\frac{\varphi(x)}{h} }\vert \ovl \partial f (x)\vert ^2 d\lambda=\Vert P_\varphi g\Vert^2_{L^2(X)} \geq
  h\int_X e^{ 2\frac{\varphi (x)}{h} }\vert  f(x)\vert^2 \triangle\varphi(x)d\lambda,
\end{equation}
and thus \eqref{gl1} holds true. Let now $\varphi(x)$ be a continuous function on $X$ such that $\triangle \varphi=\nu$
is a Borel measure on $X$. Let $\chi_\varepsilon$ be a smooth approximation of the identity.
Then $\varphi_\varepsilon= \varphi*\chi_\varepsilon$ and $\nu_\varepsilon=\nu*\chi_\varepsilon$
are well defined in a neighborhood $V$ of $Q$ for $\varepsilon$ small and one has 
$\triangle \varphi_\varepsilon d\lambda=\nu_\varepsilon$ on $V$.  The inequality \eqref{gl3} 
holds true for  $\varphi_\varepsilon $. 
Since $\varphi_\varepsilon $
converge uniformly to $\varphi$ on $Q$, it remains to verify

$$ \lim_{\varepsilon\rightarrow 0} \int_X e^{ 2\frac{\varphi_\varepsilon (x)}{h} }\vert  f(x)\vert^2
 \triangle\varphi_\varepsilon(x)d\lambda=
\int_X e^{ 2\frac{\varphi (x)}{h} }\vert  f(x)\vert^2 d\nu\ .
$$
Since the total variation on $Q$ of the  measures $\nu_\varepsilon$ is bounded,  i.e $\sup_\varepsilon\int_Q d\vert \nu_\varepsilon\vert<\infty$, one has

$$ \lim_{\varepsilon\rightarrow 0} 
\int_X \vert e^{ 2\frac{\varphi_\varepsilon (x)}{h} }-e^{ 2\frac{\varphi (x)}{h} }\vert
\vert  f(x)\vert^2
d\nu_\varepsilon = 0 \ .
$$
Then, the result follows from the convergence of the measures $\nu_\varepsilon$ to $\nu$, i.e \\
 $\lim_{\varepsilon\rightarrow 0}\int_X g(x)d\nu_\varepsilon =\int_X gd\nu $ for any continuous
function $g$ with support in $Q$.
The proof of Lemma \ref{lem1} is complete.

\end{proof}

\bigskip

Let $K\subset X$ a compact subset of $X$, 
and $\mu$ a positive Borel measure with support in $K$ such that  $\mu(K)=1$.
 We will assume that $\mu$ satisfies the following hypothesis:\\

\begin{equation}\label{gl4}
\begin{aligned}
&\text{The potential function} \ W(x)=-\frac{1}{2\pi}\int \log\vert x-y\vert d\mu(y), \ \text{solution in} \ \mathcal D'(\mathbb R^2)
\\
& \text{of the equation} \ -\triangle W= \mu, \  \ \text{is continuous on} \ \  \mathbb C\simeq \mathbb R^2.
\end{aligned}
\end{equation}

\bigskip
\noindent
For $y\in X$, we denote by $G(x,y)$  the
Green function of the Dirichlet problem

\begin{equation}\label{gl5}
\left\{  \begin{array}{ll}
-\Delta_{x} G(x,y) = \delta_{x=y} \quad \text{in } \quad  X, \\
G(.,y)\vert_{\partial X}= 0. 
\end{array} \right.
\end{equation} 
Recall that one has $G(x,y)>0$ for all $(x,y)\in X\times X$ with $x\not= y$, and that there exist  constants 
$0<c_1<c_2$
such that for all $x\in X$ with dist$(x,\partial X)$ small one has

\begin{equation}\label{gl6}
 c_1dist(x,\partial X)\leq  \inf_{y\in K}G(x,y), \quad  \sup_{y\in K}G(x,y) \leq c_2dist(x,\partial X) 
\end{equation}
Moreover, $G(x,y)$ is analytic in $x\in X\setminus \{y\}$, and more precisely one has 
\begin{equation}\label{gl7}
\left\{  \begin{array}{ll}
G(x,y)=-\frac{1}{2\pi}\int \log\vert x-y\vert+H(x,y)\\
-\Delta_x H(x,y) = 0, & \textrm{in } \ \ X, \\
H(x,y)\vert_{x\in \partial X}= \frac{1}{2\pi}\int \log\vert x-y\vert. 
\end{array} \right.
\end{equation} 
Let $\Phi_\mu (x)=\int G(x,y)d\mu(y)$. Then $\Phi_\mu$ satisfies

\begin{equation}\label{gl8}
\left\{  \begin{array}{ll}
-\Delta \Phi_\mu  = \mu, & \textrm{in } \ \ X, \\
 \Phi_\mu \vert_{\partial X}= 0. 
\end{array} \right.
\end{equation} 
By assumption \eqref{gl4}, the function  $\Phi_\mu(x)$ is continuous on $X$. Moreover, $\Phi_\mu$
is smooth on $\overline X\setminus K$, one has $\Phi_\mu(x)>0$ for all $x\in X$, and  

\begin{equation}\label{gl9}
\left\{  \begin{array}{ll}
\Phi_\mu(x)=W(x)+h(x)\\
-\Delta h = 0, \ \text{in } \  X,  \quad h\vert_{\partial X}=-W\vert_{\partial X}.
\end{array} \right.
\end{equation} 
We will denote by $C_\mu>0$ the constant
\begin{equation}\label{gl10}
C_\mu=\sup_{x\in X}\Phi_\mu(x)=\sup_{x\in K}\Phi_\mu(x)\ .
\end{equation} 
Observe that from \eqref{gl6}, one has for all $x\in X$ with dist$(x,\partial X)$ small  
\begin{equation}\label{gl6bis}
\Phi_\mu(x)\leq c_2 \text{dist}(x,\partial X).
\end{equation}

Let $Y\subset X$ an open subset of $X$ with regular boundary such that $K\subset Y$ and $\overline Y \subset X$.
We will denote by $c_Y>0$ the constant
\begin{equation}\label{gl11}
c_Y=\inf_{x\in Y, y\in K}G(x,y)\ .
\end{equation} 
From $\Phi_\mu (x)=\int G(x,y)d\mu(y)$ we get
\begin{equation}\label{gl12}
c_Y \leq \inf_{x\in Y}\Phi_\mu(x)\ .
\end{equation} 
Let $\Psi_Y(x)= -\int_Y G(x,y)d\lambda(y)$ be the solution of
\begin{equation}\label{gl13}
\triangle \Psi_Y=1_{x\in Y}, \ \   \Psi_Y\vert_{\partial X}=0 \ .
\end{equation} 
The function $ \Psi_Y$ is continuous on $\overline X$ and one has $\Psi_Y(x)\leq 0$ for all $x\in X$.
We denote by $C_Y$ the constant
\begin{equation}\label{gl14}
C_Y=\sup_{x\in Y}\vert  \Psi_Y(x)\vert  \ .
\end{equation} 
Let $\rho>0$ such that $2\rho C_Y\leq c_Y$ and define $\varphi$ by the formula

\begin{equation}\label{gl15}
\varphi(x)=\Phi_\mu(x)+\rho \Psi_Y(x) \ .
\end{equation} 
Then $\varphi$ is continuous on $X$ and one has $\triangle \varphi=-\mu+\rho1_{x\in Y} $. Thus we can apply 
\eqref{gl1} and we get for all $f\in C_0^\infty(X)$

\begin{equation}\label{gl16} 
4h^2\int_{X}  e^{ 2\frac{\varphi(x)}{h} }\vert \ovl \partial f (x)\vert ^2 d\lambda
+h \int_K e^{ 2\frac{\varphi (x)}{h} }\vert  f(x)\vert^2 d\mu 
\geq h\rho  \int_Y e^{ 2\frac{\varphi (x)}{h} }\vert  f(x)\vert^2 d\lambda.
\end{equation} 
Let $g$ be an holomorphic function in X such that $g\in L^2(X)$. We will apply \eqref{gl16} to $\psi g$, with $\psi\in C_0^\infty(X)$
such that $\psi$ is equal to $1$ in a neighborhood of $\overline Y$, and $\nabla\psi$ is supported in
$\text{dist}(x,\partial X)\leq r$ with $r>0$ small enough to have $4c_2 r\leq c_Y$. By our choice of the constants
$\rho$ and $r$,   the following inequalities hold true:

\begin{equation}\label{gl17}
\begin{aligned}
\sup_{x\in K}\varphi (x)&\leq \sup_{x\in K}\Phi_\mu(x)=C_\mu ,\\
\inf_{x\in Y}\varphi(x) &\geq \inf_{x\in Y}\Phi_\mu(x)- \rho \sup_{x\in Y}\vert  \Psi_Y(x)\vert \geq c_Y/2 ,\\
\sup_{x\in \text{support}(\nabla\psi)}\varphi (x)&\leq \sup_{dist(x,\partial X)<r}\Phi_\mu(x)\leq c_Y/4 \ .
\end{aligned} 
\end{equation} 
Since $\ovl\partial(\psi g)=g\ovl\partial\psi$, 
we get that the following Lemma holds true, with $M=\Vert \ovl\partial\psi\Vert_{L^\infty}$.
\begin{lemma}\label{lem2}
 For every holomorphic function $g\in L^2(X)$ and all $h>0$, the following inequality holds true
\begin{equation}\label{gl18} 
4h^2 M e^{ \frac {c_Y}{2h}}\int_{X}  \vert g (x)\vert ^2 d\lambda 
+h  e^{ \frac {2C_\mu}{h} }\int_K \vert  g(x)\vert^2 d\mu 
\geq h\rho e^{ \frac {c_Y}{h} } \int_Y \vert  g(x)\vert^2 d\lambda.
\end{equation} 
\end{lemma} 
\begin{remark}
Observe that from the definition \eqref{gl10} of $C_\mu$, one has obviously
$C_\mu\geq c_Y$.
\end{remark}
\begin{proposition}\label{prop1}
There exists a constant $C$, depending only on $X,Y$ such that for all holomorphic function $g\in L^2(X)$, the following interpolation inequality holds true
\begin{equation}\label{gl19}
 \int_Y \vert  g(x)\vert^2 d\lambda\leq C 
 \big( \int_K \vert  g(x)\vert^2 d\mu\big)^\delta \big( \int_{X}  \vert g (x)\vert ^2 d\lambda \big)^{1-\delta}, \quad
 \delta=\frac{c_Y}{4C_\mu-c_Y}\ .
\end{equation}
\end{proposition}
\begin{proof}
Set $\mathcal X=\int_{X}  \vert g (x)\vert ^2 d\lambda $, $\mathcal O=\int_K \vert  g(x)\vert^2 d\mu$ and
$\mathcal Y=  \int_Y \vert  g(x)\vert^2 d\lambda$. Inequality \eqref{gl18} reads
\begin{equation}\label{gl20}
 4hMe^{ \frac {-c_Y}{2h}} \mathcal X +e^{ \frac {2C_\mu-c_Y}{h} }\mathcal O\geq \rho \mathcal Y \ .
 \end{equation}
If $\mathcal O=0$, by taking the limit $h\rightarrow 0$, we get $\mathcal Y=0$, hence we  may assume $\mathcal O>0$. Let $h_*$ be such that 
$$ \frac{\mathcal X}{\mathcal O}=F(h_*), \quad F(h)=\frac{1}{4hM}e^{ \frac {4C_ \mu-c_Y}{2h} } \ .$$
If $h_*\geq 1$, we use $F(h_*)\leq F(1)$, hence $\mathcal X\leq F(1)\mathcal O$, and we write 
$$\mathcal Y\leq \mathcal X\leq F(1)^\delta \mathcal O^\delta\mathcal X^{1-\delta} .$$
Observe that $F(1)^\delta$ is independent of $C_\mu$, hence depends only on $X,Y$. \\
If $h_*\leq 1$, we use $e^{\frac{4C_\mu-c_Y}{2h_*}}\leq 4M \frac{\mathcal X}{\mathcal O}$, and we write
$$4h_*Me^{ \frac {-c_Y}{2h_*}} \mathcal X +e^{ \frac {2C_\mu-c_Y}{h_*} }\mathcal O =
2e^{ \frac {2C_\mu-c_Y}{h_*} }\mathcal O \leq (4M)^{1-\delta}\mathcal O^\delta\mathcal X^{1-\delta} .$$
The proof of Proposition \ref{prop1} is complete. 
\end{proof}

\noindent
From Proposition \ref{prop1} we deduce the following Lemma.
\begin{lemma}\label{lem3}
 Let $X\subset \mathbb C$
be a complex neighborhood of $[0,1]$. There exists  constants $C,c$ depending only on $X$
such that  the following holds true. If $E\subset [0,1]$ is a measurable set with positive measure $\vert E\vert>0$,
and $g$ a holomorphic and bounded function on $X$, one has:
\begin{equation}\label{gl21}
 \sup_{x\in [0,1]} \vert  g(x)\vert \leq \frac{C}{ \vert E\vert^{\delta/2}}
 \big( \int_E \vert  g(x)\vert^2 dx\big)^{\delta/2} \big(\sup_{x\in X} \vert g (x)\vert\big)^{1-\delta}, \quad
 \delta=\frac{c}{1+\vert\log\vert E\vert \vert}\ .
\end{equation}
\end{lemma}
\begin{proof}
We may assume $X$ bounded with regular boundary. 
We apply Proposition \ref{prop1} with $K=[0,1]$ and the measure $\mu$ defined by 
$\int gd\mu= \vert E\vert^{-1}\int_E g(x)dx$. By Cauchy integral formula, if $Y\subset\subset X$ is a complex neighborhood of $[0,1]$, there exists a constant $C$ such that 
$$   \sup_{x\in [0,1]} \vert  g(x)\vert  \leq  C ( \int_Y \vert  g(x)\vert^2 d\lambda)^{1/2}.$$
Thus, from \eqref{gl19} it just remains to verify the lower bound  on $\delta$. By formula \eqref{gl19},
this is equivalent to get a  upper bound on $C_\mu$. From formula \eqref{gl9} we get
$$C_\mu \leq C+ \sup_{x\in [0,1]} \frac{1}{\vert E\vert}\int_E \vert \log \vert x-t\vert\vert dt \leq C(1+\vert\log\vert E\vert \vert),$$  
with $C$ independent of $E$. The proof of Lemma \ref{lem2} is complete.
\end{proof}

\subsection{Interpolation estimates in $\mathbb R^d$}
Let $R>0$ be given.
Let $X\subset \mathbb C^d$ be a bounded complex neighborhood of  the closed Euclidean ball
$B_{R}=\{x\in \mathbb R^d, \ \vert x \vert \leq R\}$. Let $E\subset B_R$ be a measurable set with positive 
Lebesgue measure, $\vert E\vert>0$. The goal of this section is to prove
the following interpolation inequality.

\begin{proposition}\label{prop2}
There exists  constants $C>0,\delta\in ]0,1]$,  depending only on $X$, $R$ and $\vert E \vert$, 
such that for all holomorphic function $g\in L^2(X)$, the following interpolation inequality holds true
\begin{equation}\label{gl22}
 \int_{B_R} \vert  g(x)\vert^2 dx \leq C 
 \big( \int_E \vert  g(x)\vert^2 dx\big)^\delta \big( \int_{X}  \vert g (z)\vert ^2 \vert dz \vert  \big)^{1-\delta}\ .
\end{equation}
\end{proposition} 
\begin{proof}
The proof of Proposition \ref{prop2} is an usual consequence of Lemma \ref{lem3}.
We recall it for the reader's convenience.  We may assume $R=1$. If $\tilde X \subset \subset X$
is a complex neighborhood of $B=B_1$, by Cauchy integral formula one has
$$ \sup_{x\in \tilde X} \vert g(x)\vert \leq C ( \int_{X}  \vert g (z)\vert ^2 \vert dz \vert )^{1/2}\ .$$
Therefore, replacing $X$ by $\tilde X$, we just have to prove
\begin{equation}\label{gl23}
\sup_{x\in B} \vert g(x)\vert   \leq C 
 \big( \int_E \vert  g(x)\vert^2 dx \big)^{\delta/2}  \big(  \sup_{x\in X} \vert g(x)\vert \big)^{1-\delta}\ .
\end{equation}
Let $x_0\in B$ such that $\vert g(x_0)\vert=\sup_{x\in B} \vert g(x)\vert$. 
Let $\rho\in ]0,1[$ such that $\vert B_\rho\vert\leq \vert E\vert/2$. Set $\tilde E=E\cap \{\rho<\vert x\vert <1\}$.
One has $\vert\tilde E\vert \geq \vert E\vert/2$. Let $S^{d-1}$ be the unit sphere in $\mathbb R^d$ and
let $c_d$ be its volume. For $\omega\in S^{d-1}$, define $r_\omega\in [0,2]$ as the largest value of $r\geq 0$
such that $x_0+r\omega\in B_1$ and set 
$$ E_\omega= \{r\in ]\rho,r_\omega], \ \text{such that} \ \ x_0+r\omega\in E\} \ .$$
We denote by  $\vert E_\omega \vert \in [0,2]$ 
 the Lebesgue measure of $E_\omega\subset [0,2]$. One has  
 \begin{equation}\label{gl24}
  \vert \tilde E\vert =\int_{S^{d-1}}(\int_{r\in E_\omega}r^{d-1}dr)d\sigma(\omega)\leq 2^{d-1}
 \int_{S^{d-1}}\vert E_\omega\vert d\sigma(\omega).
 \end{equation}
 Set
 $$ V= \{\omega\in S^{d-1} \ \text{such that} \ \vert E_\omega\vert \geq \frac{\vert \tilde E\vert}{2^dc_d} \} \ .$$
 By \eqref{gl24} one has $  \vert \tilde E\vert \leq \vert \tilde E\vert/2 + 2^dVol(V)$, 
 hence $ Vol(V)\geq \vert \tilde E\vert /2^{d+1}$.\\
 Observe that there exists $a>0$ such that $X$ is a complex neighborhood of $B_{1+a}$.
 Therefore,  for each $\omega\in V$,  the function of one complex variable $z$,
 $$g_\omega(z)=g(x_0+\frac{z}{r_\omega+a}\omega)$$
   is defined in a complex neighborhood $Z$ of the interval $[0,1]$ independent of $\omega$  
   and one has $x_0+\frac{z}{r_\omega+a}\omega \in X$ for $z\in Z$. Therefore, we can apply Lemma \ref{lem2}, and we get 
 for all $\omega\in V$,
 $$\vert g(x_0)\vert^{2/\delta}= \vert g_\omega(0)\vert^{2/\delta}
 \leq C \left( \int_{E_\omega}\vert g(x_0+r\omega)\vert ^2 dr \right)  
  \big(  \sup_{x\in X} \vert g(x)\vert \big)^{\frac{2(1-\delta)}{\delta}}\ ,$$
  with $C,\delta$ depending only on $X$ and $\vert E\vert$ since with have the lower bound
$\vert E_\omega\vert \geq \frac{\vert E\vert}{c_d2^{d+1}}$.
By integration in $\omega\in V$, using $E_\omega \subset ]\rho, 2[$, we get
$$ Vol(V) \vert g(x_0)\vert^{2/\delta}\leq 
\frac {C}{\rho^{d-1}}\int_{\omega\in V}  \int_{E_\omega}\vert r^{d-1}g(x_0+r\omega)\vert^2 dr d\sigma(\omega)
  \quad \big(  \sup_{x\in X} \vert g(x)\vert \big)^{\frac{2(1-\delta)}{\delta}}\ .$$
Since 
 $$ \int_{\omega\in V}  \int_{E_\omega}\vert r^{d-1}g(x_0+r\omega)\vert^2 dr d\sigma(\omega)\leq 
 \int_E \vert g\vert^2dx \ ,$$
 we get that \eqref{gl23} holds true.
The proof of Proposition \ref{prop2} is complete. 
\end{proof}

\subsection{Proof of Proposition \ref{proposition:interpolation inequality multiD}}
In this section we  prove Proposition \ref{proposition:interpolation inequality multiD}.
Let $R, \delta>0$ given by the assumption \eqref{eq:geometricconditiontechnical}.
For $k\in \mathbb Z^d$, let $B_R(k)=k+B_R$ the closed ball of radius $R$ centered at $k$.
Increasing $R$ if necessary, we may assume that the family $(B_R(k))_{k\in\mathbb Z^d}$ is a 
covering of $\mathbb R^d$ and that we have
$$ \int_{\mathbb R^d} \vert f\vert^2dx \leq \sum_{k\in \mathbb Z^d} \int_{B_R(k)} \vert f\vert^2dx \ .$$
Let $X\subset U_{a/2}$ be a complex neighborhood of $B_R(0)$  and for any $k\in \mathbb Z^d$
set $\omega_k=\omega \cap B_R(k)$. By assumption, one has 
$\vert \omega_k\vert \geq \delta$ for all $k$. By Proposition \ref{prop2}, there exists constants 
$C, \nu>0$ independent of $k\in\mathbb Z^d$ such that 

\begin{equation}\label{gl25}
 \int_{B_R(k)} \vert  f(x)\vert^2 dx \leq C 
 \big( \int_{\omega_k}\vert  f(x)\vert^2 dx\big)^\nu \big( \int_{X+k}  \vert f (z)\vert ^2 \vert dz \vert  \big)^{1-\nu}\ .
\end{equation}
Set 
$$ c_k=  \int_{B_R(k)} \vert  f(x)\vert^2 dx, \quad a_k= \int_{\omega_k}\vert  f(x)\vert^2 dx, 
\quad b_k=  \int_{X+k}  \vert f (z)\vert ^2 \vert dz \vert $$
By H\"older's inequality with $1/p=\nu, 1/q=1-\nu$, we get
$$  \sum_k c_k\leq C\sum_k a_k^\nu b_k^{1-\nu}\leq C (\sum_k a_k)^\nu(\sum_k b_k)^{1-\nu}\ .$$
It remains to observe that one has
$$ \sum_k a_k \leq C \int_\omega \vert f(x)\vert^2 dx, 
\quad \sum_k b_k \leq C \int_{U_a}  \vert f (z)\vert ^2 \vert dz \vert  \ .$$
The proof of Proposition \ref{proposition:interpolation inequality multiD} is complete.

\section{Spectral analysis}

\label{sec: spectral analysis}

\subsection{Description of the spectrum}

The goal of this section is to give a description of the spectrum of the operator $H_{g,V}$ defined by (\ref{eq:Schrodinger operator}). To do this, we apply the long-range scattering theory developed in \cite[Chap. 30]{HormanderVol4}, which yields the following result.

\begin{proposition}
The operator $H_{g,V}$ defined by (\ref{eq:Schrodinger operator}) is self-adjoint in the  space $L^2(\R^d; \sqrt{\det g }\ud x)$. The spectrum of $H_{g,V}$ is of the form
\begin{equation}\label{prop:description of the spectrum}
\Sigma(H_{g,V}) = \Lambda \cup  \left\{ 0 \right\} \cup H_{ac},
\end{equation} where $ H_{ac} = (0,+\infty)$ is the absolutely continuous spectrum, and 
$\Lambda$ is the set of non zero eigenvalues.
Moreover, there exists $E_0 > 0$ such that  $\Lambda \subset [-E_0, 0) $
and any eigenvalue $\lambda\in \Lambda $ is isolated with finite multiplicity.
\end{proposition}

\begin{remark}
The set $\Lambda$ may be empty, or finite, or countable, and  $0$ is its only possible accumulation point.
\end{remark}

\begin{proof}

According to \cite[Sec. 30.2]{HormanderVol4}, let us split the operator $H_{g,V}$ in  the following manner 
\begin{equation}
H_{g,V} = -\Delta + \tilde{V}(x,\partial), 
\label{eq: splitting of H}
\end{equation} where $\Delta$ is the flat Laplacian in $\R^d$ and 
\begin{equation*}
\tilde{V}(x,\partial) = V(x) - \div \left[ ( g^{-1} - Id) \nabla \right] - \frac{ g^{-1} }{\sqrt{ \det g }  } \nabla ( \sqrt{\det g} ) \cdot \nabla. 
\end{equation*} Observe that we may further write 

\begin{align}
\tilde{V}(x,\partial) & = \partial_i \left[ (\delta_{ij} - g^{ij}) \partial_j \right] - \frac{ g^{-1} }{\sqrt{ \det g }  } \nabla ( \sqrt{\det g} )  \cdot \nabla + V(x)  \label{eq:splitting short long range} \\
& = V^L(x,\partial) + V^S(x,\partial), \nonumber
\end{align} with 
\begin{align}
V^L(x,\partial) & :=  (  \delta_{ij} -  g^{ij}  ) \partial^2_{i,j} - \partial_i(g^{ij})  \partial_j  + V(x), \label{eq: long range}\\
V^S(x,\partial) & := - \frac{ 1}{\sqrt{ \det g }  } g^{ij} \partial_j ( \sqrt{ \det g } ) \partial_i, \label{eq: short range}
\end{align}
 where we have used the expression in components and Einstein's convention. We have to verify that the operator $ \tilde{V} = V^S + V^L $ is a 1-admissible perturbation of the flat Laplacian
 , i.e., the short-range $ V^S(x,\partial) = \sum_{i=1}^d V^S_i (x) \partial_i $ satisfies 
\begin{equation}
V^S_i (x) \leq C_{S}(1 + |x|)^{-1 - \epsilon},  \label{eq:short range perturbation property} \\
\end{equation} and the long-range part $V^L(x,\partial)=\sum_{\vert\alpha\vert\leq 2}V_\alpha^L(x)\partial^\alpha$
satisfies
\begin{align}
& \exists \epsilon>0, \forall |\alpha| \leq 2, \quad  |\partial^{\beta} {V}^L_{\alpha}(x) | \leq C_{\alpha,\beta}(1 + |x|)^{-|\beta| - \epsilon}, |\beta| = 0,1,
\label{eq:long range perturbation property}
\end{align}
The   estimates \eqref{eq:short range perturbation property} and \eqref{eq:long range perturbation property} follows 
from \eqref{eq: metric perturbation by a symbol} and  the fact that one can write $g^{-1}=Id+\hat g$ and $\det g=1+f$
where $\hat g$ and $f$ satisfy \eqref{eq: metric perturbation by a symbol}.
Thus the operator $H_{g,V}$ is a $1$-admissible perturbation of the flat Laplacian.
As a consequence of  \cite[Theorem 30.2.10, p.295]{HormanderVol4} we get that the eigenvalues of 
the operator $H_{g,V}$ in $ \R \setminus \left\{ 0 \right\}$ are isolated with  finite  multiplicity. Furthermore, applying \cite[Theorem 1, p. 530]{SimonPaper} ensures that $H_{g,V}$ does not have eigenvalues in $\R_+$.
Finally, since for $f \in D(H_{V,g})$ one has 
\begin{equation*}
\left\langle H_{g,V} f, f \right\rangle_{L^2(\R^d; \sqrt{\det g} \ud x)}  \geq  \int_{\R^d}  V(x) |f|^2 \sqrt{\det g} \ud x 
\geq \inf_{x\in \mathbb R^d} V(x) \| f \|^2_{L^2(\R^d; \sqrt{g} \ud x)}, 
\end{equation*} one get  $\Sigma(H_{g,V})\subset [-E_0,0)$ for some $E_0>0$.

\end{proof}

\subsection{Spectral projectors and Poisson kernels}

According to the spectral theorem ( cf. \cite[Section 2.5]{Davies}) applied to the self-adjoint operator $H_{g,V}$, there exist a  measure $\ud \nu$ on $\R \times \N$, supported in $\Sigma(g,V) \times \N$, and a unitary operator 
\begin{equation*}
U : L^2(\R^d; \sqrt{\det g} \ud x ) \longrightarrow L^2(\R \times \N; \ud \nu )
\end{equation*} such that
\begin{equation}
U H_{g,V} U^{-1} (h)= \sigma h(\sigma, n), \quad h\in L^2(\R \times \N; \ud \nu ). 
\label{eq: diagonalisation}
\end{equation} 
If $F$ is a bounded Borel measurable function on $\R$, the operator
$F(H_{g,V}) $ is defined by the formula 
\begin{equation}
U F(H_{g,V}) U^{-1} = F(\sigma). 
\label{eq: calcul fonctionnel}
\end{equation}
In particular, for $\lambda\in \mathbb R$,  the spectral projector $\Pi_{\lambda}(g,V)$,
associated with the function $F(\sigma)=\mathds{1}_{\sigma<\lambda}$,
is defined by
\begin{equation}
\Pi_{\lambda}(g,V)(f)=U^{-1}\big( \mathds{1}_{\sigma<\lambda}U(f)\big).
\label{eq:spectral projectors}
\end{equation} 
For $s\geq 0$, we define the Poisson operator  $\mathbb{P}_{s,\pm}$, associated with the function
$F(\sigma)=e^{-s\sigma_\pm^{1/2}}$, by the formula 

\begin{equation}
\mathbb{P}_{s,\pm}(f)=U^{-1}\big( e^{-s\sigma_\pm^{1/2}}U(f)\big),
\label{eq:poisson operator}
\end{equation} 
where  the function $\sigma_\pm^{1/2}$ is defined in \eqref{eq: squqre root of E}.

\begin{proposition}
Let $f \in L^2(\R^d; \sqrt{\det g} \ud x)$ and set $u_\pm(s,x)=\mathbb{P}_{s,\pm}(f)(x)$.
Then $u_\pm(s,x)$ satisfies the following 
 elliptic boundary value problem
\begin{equation}
\left\{
\begin{array}{ll}
(- \partial_s^2 + H_{g,V} )u = 0, &  \textrm{in } (0, \infty) \times \R^d, \\
\lim_{s\rightarrow 0^+}u(s,x)=f(x), & \text{in}\ 
L^2(\R^d; \sqrt{\det g} \ud x). 
\end{array} \right.
\label{eq: Poisson boundary problem}
\end{equation} 
\end{proposition}

\begin{proof}
Setting 
\begin{equation*}
v(s,\sigma,n) = U(u(s, .)), \qquad g(\sigma, n) = Uf(x),
\end{equation*} the boundary value problem (\ref{eq: Poisson boundary problem}) writes 
\begin{equation}
\left\{
\begin{array}{ll}
(- \partial_s^2 + \sigma )v(s,\sigma,n) = 0, &  \textrm{in } (0, \infty) \times \R \times \N, \\
\lim_{s\rightarrow 0^+}v(s,\sigma, n)=g(\sigma,n), & \text{in}\ 
L^2(\R \times \N;  \ud \nu). 
\end{array} \right.
\label{eq: Poisson boundary problem diagonalise}
\end{equation} which is true since by construction one has 
\begin{equation*}
v(s,\sigma,n) = e^{-s \sigma_{\pm}^{\frac{1}{2}}} g(\sigma,n) . 
\end{equation*} 
\end{proof}

\section{ Analytic estimates for second order elliptic operators}
\label{sec: holomorphic extension multiD}
\subsection{Holomorphic extensions estimates}
Throughout this section we use $N\in \N$ instead of $d$ to denote the dimension of the space. Let $B_R=\{x\in \mathbb R^N, \vert x\vert <R\}$ and let $X\subset \mathbb C^N$ a complex neighborhood of 
the closed ball $\overline B_R$. Let $D_0>0, d_0>0$ be given. We denote by $\mathcal Q=\mathcal Q_{X, D_0, d_0}$ the family of  
second order differential operators $Q(x,\partial_x)$ 
of the form

\begin{equation}\label{gl30}
Q(x,\partial_x)=\sum_{\vert \alpha\vert \leq 2}q_\alpha(x)\partial^\alpha
\end{equation}
where the functions $q_\alpha(x)$ are  holomorphic in $X$, and such that
\begin{equation}\label{gl31}
\begin{aligned}
& \sup_\alpha \Vert q_\alpha\Vert_{L^\infty (X)} \leq D_0, \\
& \sum_{\vert \alpha \vert=2} \Re(q_\alpha(x))\xi^\alpha \geq d_0\vert \xi\vert^2, \quad \forall x\in X, \quad \forall \xi\in \mathbb R^N\ .
\end{aligned}
\end{equation}
Let $R'\in ]0,R[$. The goal of this section is to prove the two following propositions.
In this subsection, we will denote by $C_j$ various constants  independent of $Q\in \mathcal Q$ and of a particular solution $u\in L^2(B_R)$ of the equation $Qu=0$.

\begin{proposition}\label{prop4}
There exists a constant $C$  such that for any $Q\in \mathcal Q$
and $u\in L^2(B_R)$ such that $Qu=0$, the following inequality holds true
\begin{equation}\label{gl32}
\Vert u\Vert_{L^\infty(B_{R'})} \leq C\Vert u\Vert_{L^2(B_R\setminus B_{R'})}
\end{equation}
\end{proposition}

\begin{proposition}\label{prop5}
There exists  constants $C_j>0$  such that for any $Q\in \mathcal Q$
and $u\in L^2(B_R)$ such that $Qu=0$, the function $u$ extends as a holomorphic function in the set
$$ Y= \{z\in \mathbb C^N, \vert Re(z)\vert<R', \vert Im(z)\vert<C_1(R'-\vert Re(z)\vert) \}$$
and the following inequality holds true
\begin{equation}\label{gl33}
\sup_{z\in Y}\vert u(z)\vert \leq C_2 \Vert u\Vert_{L^2(B_R)} 
\end{equation}
\end{proposition}

\begin{remark}
It will be essential in the proof of Proposition \ref{prop6} below that the constants $C_j$
can be chosen independent of $Q\in\mathcal Q$. 
\end{remark}

\begin{proof}
The proof of Proposition \ref{prop4} is classical. Let $\varphi\in C_0^\infty (B_R)$
equal to $1$ in a neighborhood of $\overline B_{R'}$ and $\psi\in C_0^\infty (B_R)$
equal to $1$ in a neighborhood of the support of $\varphi$.  Let $s>N/2$. 
By classical pseudo-differential calculus, since $Q$ is elliptic, there exist
a pseudo-differential operator $E$ of degree $-2$ such that 

$$ EQ=\varphi+ \psi  T \psi. $$
where $T$ is a pseudo-differential operator of degree $-s$ such that
$$ \Vert \psi T(\psi f)\Vert_{H^{s}} \leq C_2 \Vert f\Vert_{L^2(B_R)}$$
Since the construction of $E,T$ involves only a finite number of derivatives of the coefficients
of $Q$, from \eqref{gl31}, the constant  $C_2$ is independent of $Q\in\mathcal Q$. From $Qu=0$,
we get $\varphi u = -\psi  T \psi u $ and therefore
\begin{equation}\label{gl34}
 \Vert \varphi u\Vert_{H^{s}} = \Vert \psi T(\psi u)\Vert_{H^{s}}
\leq C_2 \Vert u\Vert_{L^2(B_R)} \ .
\end{equation}
Let us now prove \eqref{gl32} by a contradiction argument. If \eqref{gl32} is untrue, one can find a sequence
$Q_n\in \mathcal Q$ and a sequence $u_n\in L^2(B_R)$ such that $Q_nu_n=0$, 
$\Vert u_n\Vert_{L^\infty(B_{R'})}=1$ and $\Vert u_n\Vert_{L^2(B_R\setminus B_{R'})} \rightarrow 0$.
The sequence $u_n$ is bounded in $ L^2(B_R)$ and  from \eqref{gl34}, $\varphi u_n$ is bounded
in $H^{s}$. Thus we may assume that $u_n$  weakly converge in $L^2$ to some $u\in L^2(B_R)$
and, since $s>N/2$,  that $u_n$ converge strongly to $u$ in $L^\infty(B_{R'})$.
Then we have $\Vert u \Vert_{L^\infty(B_{R'})}=1$ and $u=0$ on $B_R\setminus B_{R'}$. Let $X'$ 
an open  neighborhood of $B_R$
such that  $\overline{X'}\subset X$. We may also assume that $Q_n$ converge in $\mathcal Q(X', D_0,d_0)$
to some $Q\in \mathcal Q(X', D_0,d_0)$. Then $u$ satisfies $Qu=0$ and since $Q$ is elliptic
with analytic coefficient and $u=0$ on $B_R\setminus B_{R'}$ we get $u=0$ on $B_R$, in contradiction
with $\Vert u \Vert_{L^\infty(B_{R'})}=1$.\\

In order to prove Proposition \ref{prop5}, we will use  complex deformation arguments.\\

\noindent
We first prove that for $C_1>0$  small enough $u$ extends as a holomorphic function 
in $Y$ .
For $r \in [0,R'] $ let us define the non negative function $\psi_{r}(t)$: 
\begin{equation*}
\psi_{r}(t):= \max( R' - \sqrt{r^2 + t^2}, 0), \qquad t\in \R .
\end{equation*} The function $\psi_r$ is  Lipschitz with $\vert \psi'(r)\vert \leq 1$,  and 
$\supp \psi_{r} = \left\{ |t| \leq \sqrt{R'^2 - r^2}\right\}$. Observe that $\psi_r(t)$ is decreasing in 
$r$, $\psi_{R'}(t) = 0$,  $\psi_{0}(t) = \max( R' - |t|, 0)$.
Take $C_1>0$ small. For $r \in [0,R'] $, let $K_r$ be the compact set in $\mathbb C^N$
\begin{equation*}
K_{r} := \left\{ z \in \mathbb C^N; \vert \Im(z)\vert^2  \leq C_1\psi^2_{r}(| \Re (z)|),
 \, | \Re z| \leq \sqrt{R'^2 - r^2} \right\} \ .
\end{equation*} 
The interior $\Omega_r$ of  $K_{r}$,  is given by 
\begin{equation*}
\Omega_r := \left\{ z \in \mathbb C^N; \vert \Im(z)\vert^2 <C_1\psi^2_{r}(| \Re (z)| ),
 \, | \Re z| < \sqrt{R'^2 - r^2} \right\} \ .
\end{equation*} 
The open sets $\Omega_r$ are decreasing in $r$ and one has
$\Omega_{R'}=\emptyset$, $\Omega_0=Y$. \\
 Let $I=\{r\in [0,R'[$ such that  $u$ extends as a holomorphic function 
in $\Omega_r \}$. \\
Since $Qu=0$, the function $u$ is analytic near $\Im(z) =0$. Thus for $r$ close to 
$R'$ one has $r\in I$. From $\cup_{r>\rho}\Omega_r=\Omega_\rho$, we get that $I$ is of the form  $[r_0,R'[$.
In order to prove $r_0=0$, it is sufficient to prove that if $0<r\in I$,  $u$ extends near any point
$z_0\in K_r\setminus\Omega_r$. If $\Im (z_0)=0$, this is true since $u$ is analytic in $B_R$.
If $\Im (z_0)\not=0$, one has $ | \Re (z_0)| < \sqrt{R'^2 - r^2}$, $\vert \Im(z_0)\vert^2 = C_1\psi^2_{r}(\vert \Re z_0\vert )$,
and locally near $z_0$, $\Omega_r$ is defined by $f<0$ with 
$$ f(z)=  \vert \Im(z)\vert^2- C_1\psi^2_{r}(| \Re (z)| ) . $$
Let $\partial=\frac{1}{2}(\partial_x-i\partial_y)$. At $z_0=a+ib$, one has $b^2=C_1\psi^2_r(\vert a\vert)$ and 
$$ \partial f(z_0)=\zeta_0= \xi_0+i\eta_0, \quad \xi_0=C_1\psi_r(\vert a\vert)\frac{a}{\sqrt{r^2+a^2}},
\quad \eta_0=-b \ .$$
This implies $\vert\xi_0\vert \leq \sqrt {C_1}  \vert \eta_0\vert $.
Therefore, if $q(z,\zeta)=\sum_{\vert\alpha\vert=2}q_\alpha(z)\zeta^\alpha$ is the principal symbol of 
$Q$, one finds, using $\vert b\vert= \sqrt{C_1}\psi_r(\vert a\vert) \leq \sqrt{C_1}R'$
$$ \Re q(z_0, \zeta_0)= \Re q(a, \zeta_0)+O(\vert b\vert   \vert \zeta_0 \vert ^2)=
-\Re q(a,b)+O(\sqrt{C_1}b^2)\ .$$
By  the second line of \eqref{gl31} this implies $q(z_0, \zeta_0)\not= 0$ for $C_1$ small.
Then the result follows from the Zerner Lemma that we  recall for the reader's convenience.

\begin{lemma}[M.Zerner]
Let  $Q(z, \partial ) = \sum_{\alpha,|\alpha| \leq m} q_{\alpha} (z)\partial_z^{\alpha}$ be a linear differential operator with holomorphic coefficients defined in a neighborhood $U$ of $0$ in $\comp^N$ and let $q(z,\zeta) = \sum_{|\alpha|=m} q_{\alpha}(z) \zeta^{\alpha}$ be its principal symbol. Let $f \in \C^1( \comp^N; \R)$ be a real function such that 
$f(0) = 0$ and  $ \partial f(0) \not= 0$.  Let $u$ be a holomorphic function defined in $U \cap \left\{f  < 0 \right\}$, such that 
$Qu$ extends holomorphically to $U$. Then, if 
$q(0, \partial f (0) ) \not = 0$,  $u$ extends holomorphically near  $0$.
\label{lemma:Zerner}
\end{lemma}

Finally, let us verify that \eqref{gl33} holds true. Let $R'<R_1<R_2<R$. Let $\varphi\in C_0^\infty(B_{R_2})$,
equal to $1$ on $B_{R_1}$. Let $\delta>0$ small and $\mathcal D=\{w\in \mathbb R^N, \ \vert w\vert \leq \delta\}$.
For $w\in \mathcal D$,  we deform the real ball $B_R$ into the  countour $\Sigma_w$
$$ \Sigma_w=\{ z\in \mathbb C^d, \quad \exists x \in B_R, \quad z= x+iw \varphi(x)\} \ .$$
By the first part of the proof of Proposition \ref{prop5}, that we apply with some $R'\in ]R_2,R[$,
if $\delta$ is small enough, the function $u$ extends holomorphically near any $z\in \Sigma_w, w\in \mathcal D$.
Let $u_w(x)=u(x+iw \varphi(x))$. Then $u_w\in L^2(B_R)$ and one has $Q_w(u_w)=0$ where
$Q_w$ is the operator induced by $Q$ on $\Sigma_w$. \\
One has $Q_w\in \mathcal Q(X',D'_0,d'_0)$
with $(X',D'_0,d'_0)$ close to $(X,D_0,d_0)$ if $\delta$ is small enough.
Now, we apply Proposition \ref{prop4} with $R'=R_2$. We get in particular
$$ \sup_{w\in\mathcal D}\sup_{x\in B_{R_1}} \vert u(x+iw)\vert \leq C \Vert u\Vert_{L^2(B_R\setminus B_{R_2})} .$$
By taking $C_1$ small enough, this implies that \eqref{gl32} holds true.

\end{proof}

\subsection{Estimates on solutions of the Poisson equation}
Let $f \in L^2(\R^d; \sqrt{\det g} \ud x)$ and  $u_\pm(s,x)=\mathbb{P}_{s,\pm}(f)(x)$
the solution of the Poisson equation \eqref{eq: Poisson boundary problem}.
Let $s_0>0$.
The goal of this subsection is to prove the following proposition

\begin{proposition}\label{prop6}
There exists   constants $b>0, C>0$ independent of $f$ such  that, 
$u_\pm (s_0,.)$ extends as an holomorphic function in the set \\
$U_b=\{z\in \mathbb C^d, \vert Im(z)\vert <b\}$. Moreover $u(s_0,z)\in \mathcal H_b$ and one has
\begin{equation}\label{gl60}
\int_{z\in U_b}\vert u_\pm (s_0,z)\vert^2 \vert dz\vert \leq C\Vert f\Vert^2_{L^2}. 
\end{equation}
\end{proposition}

\begin{proof}

Recall that $u_\pm(s, x)$ is a  solution of the elliptic equation 
\begin{equation}
(- \partial_s^2 + H_{g,V} )u_\pm = 0, \quad \textrm{ in } (0, \infty) \times \R^d.
\label{eq:elliptic equation holomorphic extension}
\end{equation} 
We first choose $R\in ]0,s_0/2]$. We denote here by $B_R\subset \mathbb R^{d+1}$ the ball 
$$ B_R=\{(\sigma,x)\in \mathbb R^{d+1}, \quad \sigma^2+\vert x\vert^2 <R^2\} \ .$$
For  $w\in \mathbb R^d$, we define the function $u_w(\sigma,x)$ by the formula
$$ u_w(\sigma,x)=u_\pm(s_0+\sigma, w+x) \ .$$
Then one has $u_w\in L^2(B_R)$, and $u_w$ satisfies the equation
$$ Q_w (u_w)=0 \quad \text{on} \quad B_R $$
with $-Q_w=\tau_{(-s_0, -w)}( - \partial_s^2 +H_{g,V}) \tau_{(s_0,w)}$ where
$\tau_{(s_0,w)}$ is the translation by $(s_0,w)$. By hypothesis \eqref{eq: metric and potential holomorphically extended} and \eqref{eq: metric perturbation by a symbol} there exist $(X,D_0,d_0)$ such that 
$Q_w\in \mathcal Q(X,D_0,d_0)$ for all $w$. Let $R'<R$. By Proposition \ref{prop5}, there exist $b>0$ and $C>0$
independent of $w\in \mathbb R^d$ such that
$$ \sup_{\vert x\vert <R', \vert y\vert <b} \vert u_w(0, x+iy) \vert^2\leq 
C\int_{B_R} \vert u_w(\sigma,x) \vert^2 d\sigma dx\ .$$
This implies

\begin{equation}\label{gl61}
\sup_{\vert y\vert <b} \int_{B_{R'}} \vert u_\pm(s_0, w+x+ iy)\vert ^2 dx \leq
C\int_{B_R} \vert u_\pm(s_0+\sigma, w+x) \vert^2 d\sigma dx.
\end{equation}
Applying \eqref{gl61} at points $w_k= hk, \ k\in \mathbb Z^d$ with $h$ small enough
and adding all these inequalities, we get with a different constant $C$
 
 \begin{equation}\label{gl62}
\sup_{\vert y\vert <b} \int_{\mathbb R^d} \vert u_\pm(s_0, x+ iy)\vert ^2 dx \leq
C\int_{\mathbb R^d}\int_{s\in [s_0-R',s_0+R']} \vert u_\pm(s, x) \vert^2 ds dx.
\end{equation}
This proves $u_\pm(s_0,z)\in \mathcal H_b$ and \eqref{gl60} follows from
$\Vert  u_\pm(s, .) \Vert_{L^2(\mathbb R^d)}\leq \Vert f\Vert_{L^2(\mathbb R^d)}$ for all $s>0$.
\end{proof}

\section{Proof of Theorem \ref{thm: Spectral inequality} }
\label{sec: End of the proof multiD}

Let $\mu \in \R$ and let $f \in L^2(\R^d;\sqrt{\det g} \ud x)$ be such that $f = \Pi_{\mu}(g,V)f$. 
Take  $s_0>0$. Since the support of $U(f)$ is contained in $\sigma < \mu$, we can define the function 
$$h=U^{-1}\big( e^{s_0\sigma_\pm^{1/2}}U(f)\big) $$
which satisfies
$$
\Vert h\Vert_{L^2}\leq \vert e^{s_0 \mu_\pm ^{1/2}}\vert \Vert f\Vert_{L^2} \ .
$$
The function  $u_\pm (s,.) = \mathbb{P}_{s,\pm}(h)$ is solution of the Poisson equation with data $h$ on 
$s=0$ and one has by construction 
\begin{equation}
f=u_\pm (s_0, .) \ .
\end{equation} 
By Proposition \ref{prop6}, there exists $b>0$ such that $f\in \mathcal H_b$ and one has

\begin{equation}\label{gl100}
\int_{z\in U_b}\vert f(z)\vert^2 \vert dz\vert \leq C\Vert h\Vert^2_{L^2} \leq C 
\vert e^{2s_0 \mu_\pm ^{1/2}}\vert \Vert f\Vert^2_{L^2}.
\end{equation}
By Proposition \ref{proposition:interpolation inequality multiD}, we get that there exists $\nu>0$ such that 

\begin{equation}\label{gl101}
\int_{\R^d} |f|^2 \ud x \leq C \left(  \int_{\omega} |f|^2 \ud x \right)^{\nu} \left(  \int_{U_b} |f|^2 |\ud z| \right)^{1 - \nu},
\end{equation} 
From \eqref{gl100} and \eqref{gl101} we get

$$ \int_{\R^d} |f|^2 \ud x \leq C \vert e^{\frac{2(1-\nu)s_0 \mu_\pm ^{1/2}}{\nu}}\vert  \int_{\omega} |f|^2 \ud x \ .$$
The proof of Theorem \ref{thm: Spectral inequality}  is complete.

\bibliographystyle{plain}                            
{}

\end{document}